\documentclass{amsart}
\usepackage{amsfonts,amssymb}
\usepackage{enumerate,graphicx}
\usepackage{caption}
\usepackage{subcaption}
\usepackage[square,numbers]{natbib}
\usepackage{mathrsfs,bbm}
\usepackage{tikz,tikzscale,tikz-cd}
\usepackage{multirow,xparse,fp}
\usepackage{environ}
\usepackage{amstext}
\usepackage{hyperref}
\usepackage{diagbox}
 \usepackage[pagewise]{lineno}
\usepackage{comment}
\newtheorem{theorem}{Theorem}[section]

\theoremstyle{definition}

\newtheorem{remark}[theorem]{Remark}

\allowdisplaybreaks[4]

\begin{document}
\title{The entropy structures of axial products on $\mathbb{N}^d$ and Trees}

\author[Jung-Chao Ban]{Jung-Chao Ban}
\address[Jung-Chao Ban]{Department of Mathematical Sciences, National Chengchi University, Taipei 11605, Taiwan, ROC.}
\address{Math. Division, National Center for Theoretical Science, National Taiwan University, Taipei 10617, Taiwan. ROC.}
\email{jcban@nccu.edu.tw}

\author[Wen-Guei Hu]{Wen-Guei Hu}
\address[Wen-Guei Hu]{College of Mathematics, Sichuan University, Chengdu, 610064, China}
\email{wghu@scu.edu.cn}

\author[Guan-Yu Lai]{Guan-Yu Lai}
\address[Guan-Yu Lai]{Department of Mathematical Sciences, National Chengchi University, Taipei 11605, Taiwan, ROC.}
\email{gylai@nccu.edu.tw}

\keywords{Full axial extension, entropy, multiplicative integer systems}

\thanks{Ban is partially supported by the Ministry of Science and Technology, ROC (Contract MOST 111-2115-M-004-005-MY3). Hu is partially supported by the National Natural Science Foundation of China (Grant 11601355). Lai is partially supported by the Ministry of Science and Technology, ROC (Contract MOST 111-2811-M-004-002-MY2).}


\baselineskip=1.2\baselineskip

\begin{abstract}
 In this paper, we first concentrate on the possible values and dense property of entropies for isotropic and anisotropic axial products of subshifts of finite type (SFTs) on $\mathbb{N}^d$ and $d$-tree $\mathcal{T}_d$. We prove that the entropies of isotropic and anisotropic axial products of SFTs on $\mathbb{N}^d$ are dense in $[0,\infty)$, and the same result also holds for anisotropic axial products of SFTs on $\mathcal{T}_d$. However, the result is no longer true for isotropic axial products of SFTs on $\mathcal{T}_d$. Next, motivated by the work of Johnson, Kass and Madden \cite{johnson2007projectional}, and Schraudner \cite{schraudner2010projectional}, we establish the entropy formula and structures for full axial extension shifts on $\mathbb{N}^d$ and $\mathcal{T}_d$. Combining the aforementioned results with the findings on the surface entropy for multiplicative integer systems \cite{ban2022boundary} on $\mathbb{N}^d$ enables us to estimate the surface entropy for the full axial extension shifts on $\mathcal{T}_d$. Finally, we extend the results of full axial extension shifts on $\mathcal{T}_d$ to general trees. 
\end{abstract}
\maketitle

\tableofcontents
\section{Introduction}

Let $\mathcal{A}$ be a finite set and $X_{1},\ldots
,X_{d}\subseteq \mathcal{A}^{\mathbb{N}}$ be $d$ shifts, the associated \emph{axial
	product of subshifts }$X_{1},\ldots ,X_{d}$\emph{\ on }$\mathbb{N}^{d}$,
write $\otimes _{i=1}^{d}X_{i}=X_{1}\otimes \cdots \otimes X_{d}\subset 
\mathcal{A}^{\mathbb{N}^{d}}$ is defined as follows. 
\[
\otimes _{i=1}^{d}X_{i}=\{x\in \mathcal{A}^{\mathbb{N}^{d}}:\forall g\in 
\mathbb{N}^{d}\text{ }\forall i\in \{1,\ldots ,d\}\text{, }x_{g+\mathbb{Z}%
	e_{i}}\in X_{i}\}\text{,} 
\]%
where $x_{g+\mathbb{Z}e_{i}}\in \mathcal{A}^{\mathbb{N}}$ is the sequence
obtained by shifting $x$ by $g$ and $\{e_{1},\ldots ,e_{d}\}$ denotes the standard basis of $\mathbb{N}^{d}$. Suppose $\mathcal{T}_{d}$ is a conventional $d$%
-tree, that is, $\mathcal{T}_{d}$ is a free semigroup generated by $\Sigma
=\{f_{1},\ldots ,f_{d}\}$ with the root, say $\epsilon $. The \emph{axial product of subshifts }$X_{1},\ldots ,X_{d}\subseteq \mathcal{%
	A}^{\mathbb{N}}$\emph{\ on }$\mathcal{T}_{d}$, write $\times
_{i=1}^{d}X_{i}=X_{1}\times \cdots \times X_{d}$ is defined similarly. That
is, 
\[
\times _{i=1}^{d}X_{i}=\{x\in \mathcal{A}^{\mathcal{T}_{d}}:\forall g\in 
\mathcal{T}_{d}\text{ }\forall i\in \{1,\ldots ,d\}\text{, }x_{g+\mathbb{Z}%
	f_{i}}\in X_{i}\}\text{.}
\]%
An axial product $\otimes _{i=1}^{d}X_{i}$ (or $\times _{i=1}^{d}X_{i}$) is
called \emph{isotropic} if $X_{i}=X_{j}$ $\forall 1\leq i\neq j\leq d$, and
is called \emph{anisotropic }if it is not isotropic\footnote{%
	An isotropic axial product space is also called a \emph{hom-shifts} on $\mathbb{N}^{d}
	$ \cite{chandgotia2016mixing} or on $\mathcal{T}_{d}$ \cite{petersen2020entropy, petersen2021asymptotic}}. 

The isotropic axial product of
shifts on $\mathbb{N}^{d}$ is introduced in \cite{louidor2013inde}. Many important physical
systems, e.g., the hard square model on $\mathbb{N}^{2}$ (or $\mathbb{Z}^{2}$%
), are characterized by this kind of multidimensional shift. In \cite{chandgotia2016mixing},
the authors study the decidability for some topological properties of $%
\otimes _{i=1}^{d}X_{i}$ for $d\geq 2$. In \cite{louidor2013inde,meyerovitch2014independence}, the authors study the relation between
the limiting and independence entropy of $\otimes _{i=1}^{d}X_{i}$ as $%
d\rightarrow \infty $.

The axial product space on $\mathcal{T}_{d}$ is a sort of tree-shift (cf. 
\cite{ban2017mixing, ban2017tree, PS-2017complexity, petersen2020entropy,
	petersen2021asymptotic, aubrun2012tree}), and it is attracting a lot of attention
recently since $\mathcal{T}_{d}$ is not an amenable group and the shifts
defined on it exhibit very rich and different phenomena in the topological
(cf. \cite{ban2017mixing}) and statistical prospects (cf. \cite%
{ban2021structure, ban2022topological, petersen2021asymptotic}). In \cite{petersen2021asymptotic}, the authors extend the concept of limiting entropy \cite{meyerovitch2009growth, meyerovitch2014independence} on $\mathbb{N}^d$ to asymptotic pressure on $\mathcal{T}_d$ and study its limiting behavior. The aim of this paper is to
investigate the entropy structures of $\otimes _{i=1}^{d}X_{i}$ and $\times
_{i=1}^{d}X_{i}$, and we introduce the formal definition for the topological
entropy below.

Let $F\subseteq \mathbb{N}^{d}$ be a finite set, we denote by $\mathcal{P%
}(F,X):\mathcal{A}^{\mathbb{N}^{d}}\rightarrow \mathcal{A}^{F}$ the \emph{%
	canonical projection }of $X\subseteq \mathcal{A}^{\mathbb{N}^{d}}$ into $%
\mathcal{A}^{F}$, i.e., $\mathcal{P}(F,X)=\{(x_{g})_{g\in F}\in \mathcal{A}%
^{F}:x\in X\}$. Denote $F_{n}:=[1,n]^{d}$ and for a subshift $X\subset \mathcal{A}^{%
	\mathbb{N}^{d}}$, the \emph{entropy }of $X$ is defined as 
\begin{equation}
h(X)=\lim_{n\rightarrow \infty }\frac{\log \mathcal{P}(F_{n},X)}{%
	\left\vert F_{n}\right\vert }\text{,}  \label{1}
\end{equation}%
where $\left\vert F\right\vert $ denotes the number of the elements in $F$.
For a subshift $X\subset \mathcal{A}^{\mathcal{T}_{d}}$, let 
\[
\Delta _{n}=\{g\in \mathcal{T}_{d}:\left\vert g\right\vert \leq n\}\text{
	and }\mathcal{T}_{n}=\{g\in \mathcal{T}_{d}:\left\vert g\right\vert =n\}%
\text{,}
\]%
then the \emph{topological entropy} of $X\subseteq \mathcal{A}^{\mathcal{T}%
	_{d}}$ is defined similarly. 
\begin{equation}
h^{\mathcal{T}}(X)=\lim_{n\rightarrow \infty }\frac{\log \mathcal{P}(\Delta
	_{n},X)}{\left\vert \Delta _{n}\right\vert }\text{.}  \label{2}
\end{equation}%
The primary objective of this paper is to study the entropy structures of $%
h(\otimes _{i=1}^{d}X_{i})$ and $h^{\mathcal{T}}(\times _{i=1}^{d}X_{i})$.
The limit (\ref{1}) exists since $\mathbb{N}^{d}$ is an amenable group (cf. 
\cite{ceccherini2010cellular}). That is, if $F_{n}$ is a \emph{F\o lner sequence}, i.e., $\lim_{n \rightarrow \infty }\frac{\left\vert gF_{n}\Delta	F_{n }\right\vert }{\left\vert F_{n }\right\vert }=0$, then the limit 
\[
h(X)=\lim_{n \rightarrow \infty }\frac{\log \mathcal{P}(F_{n },X)}{%
	\left\vert F_{n}\right\vert }
\]%
exists and is equal to (\ref{1}). The existence of the limit (\ref{2}) is
proved in \cite{PS-2017complexity}\footnote{%
	Since the subadditive property does not hold true for shifts on $\mathcal{T}%
	_{d}$, the proof of the existence of (\ref{2}) is quite different than that in the
	cases where shifts on $\mathbb{N}^{d}$. We refer the reader to \cite{ban2021stem} for
	the existence of the limit (\ref{2}) for shifts defined in a large class of
	trees.}, and it appears that the structures of $\{h(X):X$ is a
tree-SFT$\}$\footnote{For the simplicity of notations, we denote $h^\mathcal{T}$ by $h$.} and $\{h(X):X$ is an $\mathbb{N}^{d}$ SFT$\}$ are quite
different (cf. \cite{ban2021structure}). This motivates a systematic study
on the sets 
\[
\{h(\otimes _{i=1}^{d}X_{i}):X_{i}\subseteq \mathcal{A}^{\mathbb{N}}\text{ }%
\forall i=1,\ldots ,d\}\text{ and }\{h(\times
_{i=1}^{d}X_{i}):X_{i}\subseteq \mathcal{A}^{\mathbb{N}}\text{ }\forall
i=1,\ldots ,d\}\text{.}
\]%

The study consists of four components, and we present the motivations and
findings below.

1). The entropy structure of the axial product on $\mathbb{N}^{2}$ and $%
\mathcal{T}_{2}$. It is known that the set $\{h(X):X\in \mathcal{A}^{\mathbb{%
		Z}}$ is an mixing SFT$\}$ is the logarithm of the numbers in the spectral
radii of aperiodic non-negative integral matrices \cite{Lind-ETaDS1984}, and
the set $\{h(X):X\in \mathcal{A}^{\mathbb{Z}^{d}}$ is an SFT$\}$, $%
d\geq 2,$ is the class of non-negative right recursively enumerable numbers 
\cite{HM-AoM2010a}. The above results focus on all possible values of the
set of entropies of SFTs on $\mathbb{N}^{d}$, $d\geq 1$ in algebraic and
computational perspectives. Besides, Desai \cite{desai2006subsystem} proved
that any $\mathbb{Z}^{d}$ SFT (resp. sofic) $X$ with $h(X)>0$ contains a
family of $\mathbb{Z}^{d}$ subSFTs (subsofics) with entropies dense in the
interval $[0,h(X)]$. This indicates that the possible entropies of $\mathbb{Z%
}^{d}$ SFTs (or sofic) dense in $[0,\infty )$.

In this paper, we concentrate on the possible values and dense property of entropies for
isotropic and anisotropic axial products of SFTs on $\mathbb{N}_{d}$ and $%
\mathcal{T}_{d}$. For simplicity, we consider the cases where $d=2$, and the
case in which $d\geq 2$ can be dealt with in the same manner. Theorem \ref{thm 2.1} demonstrates that the set of entropy of anisotropic (or isotropic) axial products of subshifts of finite type (SFTs) on $\mathbb{N}%
^{2}$ is dense in $[0,\infty )$. The same result holds true for the
anisotropic axial products of SFTs on $\mathcal{T}_{2}$\emph{. }However,
this outcome no longer holds true for isotropic axial products of SFTs on $%
\mathcal{T}_{2}$, Theorem \ref{thm 2.1} reveals that the closure of possible entropy
values of isotropic axial products of SFTs cannot intersect the interval $(0,%
\frac{\log 2}{2})$. This phenomenon is new and differs from the case of the
axial products on $\mathbb{N}^{2}$. The possible reason lies in the inherent
difference in the structure of the $\mathbb{N}^{2}$ and $\mathcal{T}_{2}$
and the isotropic constrain.

2). Whether transitivity implies the positivity of the entropy. Another
viewpoint comes from \cite{kolyada1997some}, in Section 9 of \cite%
{kolyada1997some}, the authors describes the following:

\begin{quote}
	The question whether transitivity implies the positivity of the entropy is
	challenging. Moreover, if the answer is affirmative, one can ask what is the
	best lower bound for the entropy of transitive maps in the space under  consideration.
\end{quote}

Indeed, there are spaces in which transitive maps have zero topological
entropy, and there is also a class of transitive maps defined in the interval
or circle in which the infimum of the entropies of the class is positive
(see \cite{kolyada1997some} and the references given therein). Thanks to the
aforementioned result, Theorem \ref{thm inf} reveals that the infimum of the entropies
of the set $\{\times _{i=1}^{d}X_{i}:\times _{i=1}^{d}X_{i}$ is transitive
and has a periodic point$\}$ (or $\{\otimes _{i=1}^{d}X_{i}:\otimes
_{i=1}^{d}X_{i}$ is transitive and has a periodic point$\}$) is zero.

3). The entropy structure of full axial extension shifts on $\mathbb{N}^{d}$
and $\mathcal{T}_{d}$. Let $E=\mathcal{A}^{\mathbb{N}}$, that is, the full
shift. It is not difficult to show that the entropy of the axial product of $E$
with a shift $X$ on $\mathbb{N}^2$ equals the entropy of $X$, i.e., $h(E\otimes X)=h(X)$. The
converse also holds true when the underlined lattice is $\mathbb{Z}^{2}$
(Theorem 4.1. \cite{johnson2007projectional}) or $\mathbb{Z}^{d}$ in which $X$ possesses some mixing assumption (Theorem
2.3. \cite{schraudner2010projectional}). However, the property is not generally
true for the axial product spaces defined in $\mathcal{T}_{d}$. Therefore,
to calculate the explicit entropy values for $E\times X$ on $\mathcal{T}_2$ and $E^{r}\times
_{i=1}^{d-r}X_{i}:=\left( E^{\times _{r}}\right) \times _{i=1}^{d-r}X_{i}$ on $\mathcal{T}_d$ is interesting and is presented in Theorem \ref{thm 3.1}, where $\left( E^{\times
	_{r}}\right) $ stands for the $r$-axial product of $E$, for $r\in\mathbb{N}$. In addition, we
focus on the lower bound for the values of $\{h(E^{r}\times
_{i=1}^{d-r}X_{i}):X_{i}$ is an SFT $\forall 1\leq i\leq d-r\}$ as well.
Theorem \ref{thm 4.2} reveals that different types, e.g., nonempty, general subshifts or SFTs, of $\times_{i=1}^{d-r}X_i\subseteq 
\mathcal{A}^{\mathcal{T}_d}$ give rise to different lower bounds of $%
h\left( E^{r}\times _{i=1}^{d-r}X_{i}\right) $. We stress that this
behavior cannot happen in the full axial extension shifts on $\mathbb{N}^{d}$.
Using the results of surface entropy of multiplicative systems and full
axial product on $\mathcal{T}_{d}$, we give the surface entropy of full
axial extension shifts on $\mathcal{T}_{d}$ (see Section \ref{sec surface} for details).

4). General trees. The aforementioned results focus on the axial product of
shifts on $\mathbb{N}^{d}$ and $\mathcal{T}_{d}$. The natural question is,
can these outcomes be extended to broader lattices? Theorem \ref{thm 5.1} gives the explicit formula for $h(E\times X)$ and $h(X\times E)$, where $E\times X$ and $X\times E$ are the full
axial extensions on golden-mean tree $\mathcal{G}$, and it appears that both values do not generally coincide. The \emph{golden-mean tree} $%
\mathcal{G}$ is a kind of Markov-Cayley tree (defined in Section \ref{sec 5}) with the
adjacency matrix is $M=\left[\begin{array}{cc}
1 & 1 \\ 
1 & 0%
\end{array}%
\right] $. Generally, let $\mathcal{T}_{n}=\{g\in \mathcal{T}:\left\vert g\right\vert
=n\}$, and $\gamma =\lim_{n\rightarrow \infty }\frac{|\mathcal{T}_{n+1}|}{%
	|\mathcal{T}_{n}|}$, Theorem \ref{thm general} demonstrates that if $\gamma >1$, then for all axial product $X_{1}\times X_{2}$ we have $h(X_{1}\times X_{2})>
h(X_{2})$. However, if $\gamma =1$, there exists an example of axial
product $X_1\times X_2$ on which $h(X_{1}\times X_{2})=h(X_{2})$.
This can be seen an analogous result as that of Theorem \ref{thm 3.1} for the cases where the underlined space is $\mathbb{N}^d$.

\section{The entropy structure of the axial product on $\mathbb{N}^{2}$ and $\mathcal{T}_{2}$}

\subsection{Entropy structure}
Let $\mathbb{N}$ be the set of positive integers. For $d\in\mathbb{N}$, $\mathbb{N}^d=\{(n_1,...,n_d):n_i\in\mathbb{N}\mbox{ for all }1\leq i \leq d\}$, and $d$-tree $\mathcal{T}_d$ is a semigroup which is generated by $d$ generators.

Let $1\leq r \leq d$. Define the set of entropies of $d$ dimensional full $r$ extension on $d$-tree by
\[
\mathcal{H}^{d,r}=\left\{ h(E^r\times_{i=1}^{d-r} X_i): X_1,...,X_{d-r}\mbox{ are SFTs} \right\},
\]
and if $X_i=X$ for all $1\leq i \leq d-r$, we denote $\mathcal{H}^{d,r}$ by $\mathcal{H}_i^{d,r}$. Similarly, we define the set of entropies of $d$ dimensional full $r$ extension on $\mathbb{N}^d$ by $\mathbf{H}^{d,r}$ and if $X_i=X$ for all $1\leq i \leq d-r$, we denote $\mathbf{H}^{d,r}$ by $\mathbf{H}_i^{d,r}$. For simplicity, we only consider $d=2$ and $r=0$. The following notations are also needed. Let $\mathbb{Z}_{m\times n}=\{1,...,m\}\times \{1,...,n\}$, and $\mathbb{Z}_n=\{1,...,n\}$, $\Delta_n=\{g\in\mathcal{T}: |g|\leq n\}$ and $\Delta_0=\{\epsilon\}$.

\begin{theorem}\label{thm 2.1}
	The closures of $\mathbf{H}^{2,0}_i,\mathbf{H}^{2,0}$ and $\mathcal{H}^{2,0}$ equal $[0,\infty)$. Moreover, the closure $\mathcal{H}^{2,0}_i$ does not intersect $(0,\frac{\log 2}{2})$. 
\end{theorem}	

\begin{proof}
		We first prove that the closure of $\mathbf{H}^{2,0}_i$ equals $[0,\infty)$. Let $X=X_A$ be an SFT with transition matrix $A$, where $A=[A_{i,j}]$ is a $2^m+n-1$ by $2^m+n-1$ matrix and
	\begin{equation*}
		A_{i,j}=\left\{\begin{array}{ll}
			1&\mbox{, if }1\leq i \leq 2^m \mbox{ and }j=2^m+1,\\
			1&\mbox{, if }2^m+1\leq i \leq 2^m+n-2\mbox{ and }j=i+1, \\
			1&\mbox{, if }i=2^m+n-1\mbox{ and } 1\leq j \leq 2^m,\\
			0&\mbox{, otherwise.}
		\end{array} \right.
	\end{equation*}
	See Figure \ref{mncycle} for the directed graph of the transition matrix $A$.
\begin{figure} 
	\centering 
	\includegraphics[width=0.6\textwidth]{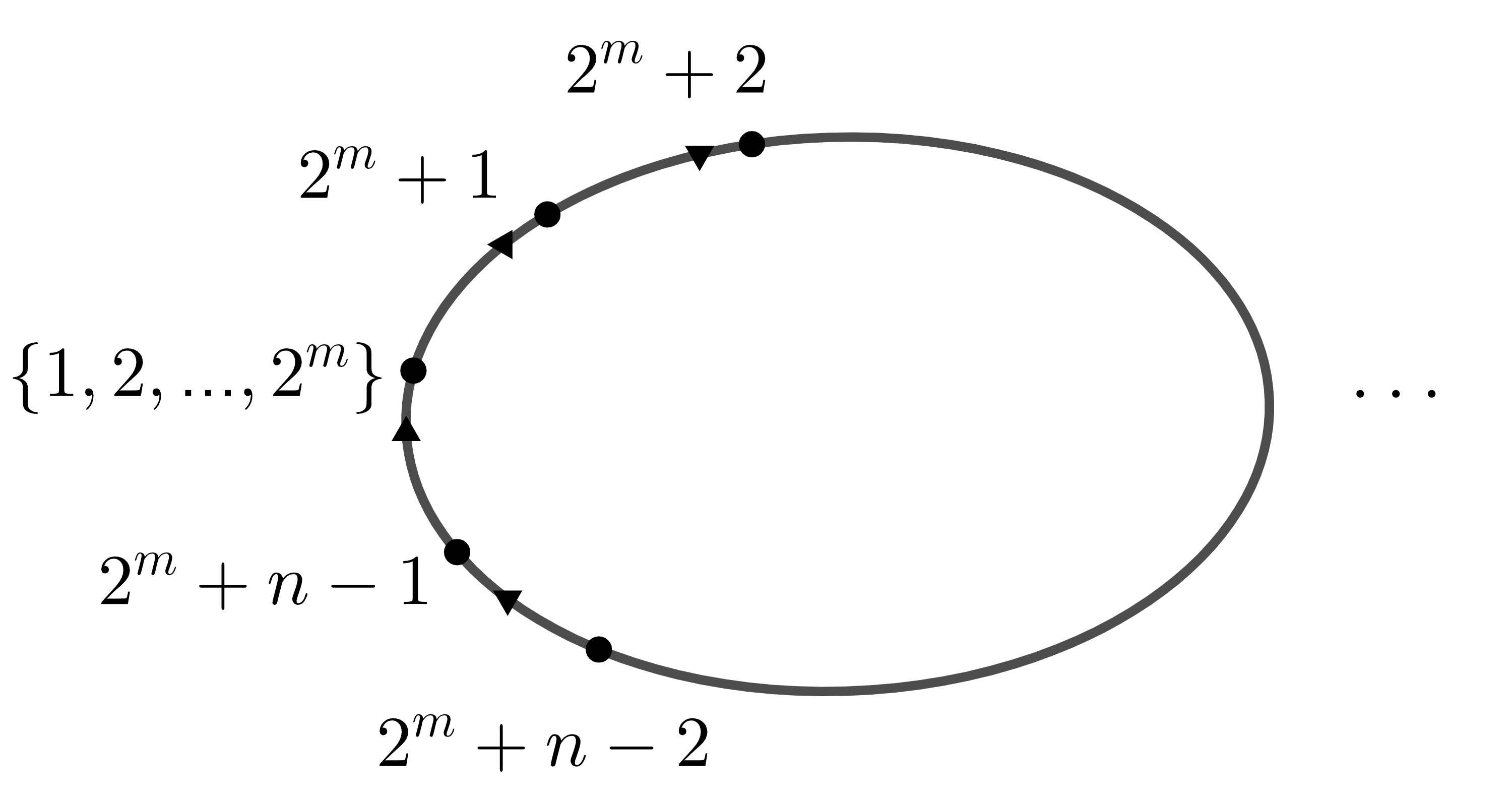} 
	\caption{The directed graph of the transition matrix $A$.} 
	\label{mncycle} 
\end{figure}

In the following, we claim that for $k\geq 1$,
	\begin{equation*}
		\left|\mathcal{P}(\mathbb{Z}_{kn\times k},X\otimes X)\right|=n(2^m)^{k^2}.
	\end{equation*}
 Indeed, due to the rule of $A$, when the symbol on $\mathbb{Z}_{1\times 1}$ is fixed then the positions of symbols $1\leq i\leq 2^m$ are uniquely determined. Thus, if the symbol on $\mathbb{Z}_{1\times 1}$ is in $\{2^m+1,...,2^m+n-1\}$, then we have $k^2$ positions on $\mathbb{Z}_{kn\times k}$ having $2^m$ choices. If the symbol on $\mathbb{Z}_{1\times 1}$ is in $\{1,...,2^m\}$, then we have $k^2-1$ positions on $\mathbb{Z}_{kn\times k}$ having $2^m$ choices. Hence,
 \begin{align*}
 \left|\mathcal{P}(\mathbb{Z}_{kn\times k},X\otimes X)\right|&=|\{2^m+1,...,2^m+n-1\}|(2^m)^{k^2}+|\{1,...,2^m\}|(2^m)^{k^2-1}\\
 &=(n-1)(2^m)^{k^2}+(2^m)(2^m)^{k^2-1}\\
 &=n(2^m)^{k^2}.
 \end{align*}
 The proof of claim is complete.

By the existence of the entropy of $X\otimes X$ and the above claim, we have
	\begin{align*}
		h(X\otimes X)&=\lim_{k,\ell \to \infty}\frac{\log|\mathcal{P}(\mathbb{Z}_{k\times \ell},X\otimes X)|}{k\ell}\\
		&=\lim_{k \to \infty}\frac{\log|\mathcal{P}(\mathbb{Z}_{kn\times k},X\otimes X)|}{k^2n}\\
		&=\lim_{k \to \infty}\frac{\log n+ k^2\log 2^m}{k^2n}\\
		&=\frac{m\log 2}{n}.
	\end{align*}
	Thus, the closure of $\{h(X\otimes X)=\frac{m\log 2}{n}: m,n\in\mathbb{N}\}\subseteq\mathbf{H}^{2,0}_i$ is $[0,\infty)$. Since $h(X\otimes X)\geq 0$ for all SFT $X$, we have the closure of $\mathbf{H}^{2,0}_i$ is $[0,\infty)$.
	
	Secondly, we prove that the closure of $\mathbf{H}^{2,0}$ equals $[0,\infty)$. By Theorem \ref{thm 3.1}, we have $\{h(X_2):X_2\mbox{ is an SFT}\}\subseteq \mathbf{H}^{2,0}$.
	Since the closure of $\{h(X_2):X_2\mbox{ is an SFT}\}$ is $[0,\infty)$, and $h(X_1\otimes X_2)\geq 0$ for all SFTs $X_1$ and $X_2$, we have the closure of $\mathbf{H}^{2,0}$ is $[0,\infty)$. 
	
	Thirdly, we prove that the closure of $\mathcal{H}^{2,0}$ equals $[0,\infty)$. Let $X_1=X_A$ be an SFT with transition matrix $A$ as above, and $X_2$ be an identity matrix. For $k\geq 1$, let $\mathcal{T}_k=\{g\in \mathcal{T}_2: |g|=k\}$, then we have the following recurrence relation: 
	
	For $k\geq 1$ and $0\leq j \leq n-1$, let $a^{\{1,...,2^m\}}_k(j)$ be the number of positions (labeled the symbol in $\{1,...,2^m\}$) on $\mathcal{T}_k$ and $a_k^{i}(j)$ be the number of positions (labeled the symbol $i, 2^m+1\leq i\leq 2^m+n-1$) on $\mathcal{T}_k$. We have 
	\begin{equation}\label{recurrence 1}
	\left[\begin{matrix}
	a^{\{1,...,2^m\}}_k(j)\\
	a^{2^m+1}_k(j)\\
	\vdots\\
	a^{2^m+n-1}_k(j)
	\end{matrix}\right]=\left[\begin{matrix}
	1&& &&1\\
	1&1&\\
	&\ddots&\ddots\\
	&&\ddots&1&\\
	&&&1&1	
	\end{matrix}\right]_{n\times n}\left[\begin{matrix}
	a^{\{1,...,2^m\}}_{k-1}(j)\\
	a^{2^m+1}_{k-1}(j)\\
	\vdots\\
	a^{2^m+n-1}_{k-1}(j)
	\end{matrix}\right]
	\end{equation}
	with initial vector $(a_0^{\{1,...,2^m\}}(j),...,a_0^{2^m+n-1}(j))^t=e_{j+1}$ and $\{e_1,...,e_n\}$ is standard basis in $\mathbb{R}^n$. See Figure \ref{recurrsive} for the recursive relation (\ref{recurrence 1}).	\begin{figure} 
		\centering 
		\includegraphics[width=1\textwidth]{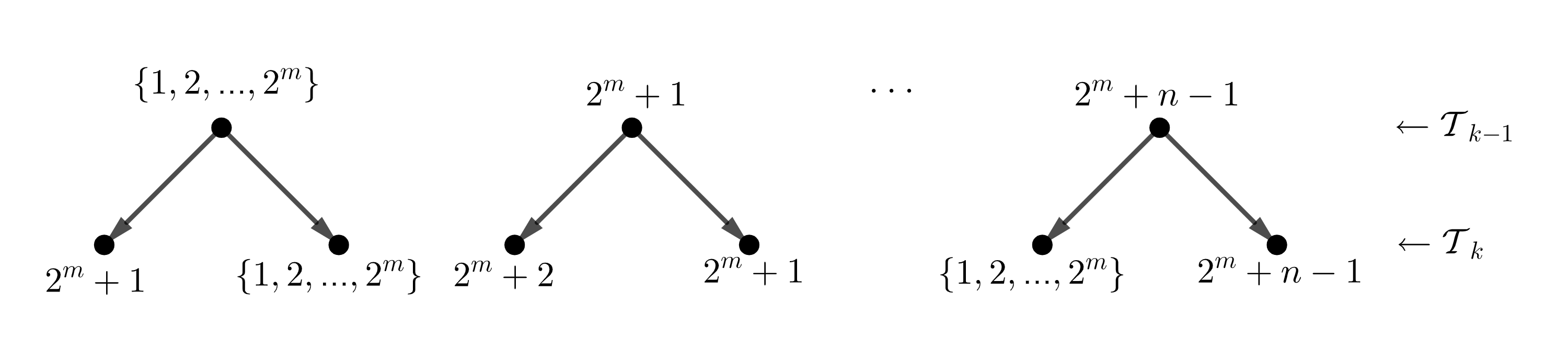} 
		\caption{The recurrsive relation of $a^i_k(j)$ from $\mathcal{T}_{k-1}$ to $\mathcal{T}_k$.} 
		\label{recurrsive} 
	\end{figure}

	Then we have
	\begin{equation*}
		\left|\mathcal{P}\left(\Delta_k,X_1\times X_2\right)\right|=\sum_{j=0}^{n-1}(2^m)^{a_k^{\{1,...,2^m\}}(j)}.
	\end{equation*}
	
	By Theorem 4.5.12 \cite{LM-1995} with left and right eigenvectors $v=w^t=\frac{1}{\sqrt{n}}(1,...,1)$ corresponding to the maximum eigenvalue $2$, after computation, we have $a_k^{\{1,...,2^m\}}(j)=\frac{2^k}{n}+o(2^k)$ for all $0\leq j \leq n-1$. This implies   
	\begin{align*}
		h(X_1\times X_2)&=\lim_{k\to\infty}\frac{\log\left|\mathcal{P}\left(\Delta_k,X_1\times X_2\right)\right|}{1+2+\cdots+2^k}\\
		&=\lim_{k\to\infty}\frac{\log\sum_{j=0}^{n-1}(2^m)^{a_k(j)}}{2^{k+1}-1}\\
		&=\lim_{k\to\infty}\frac{\log n+\left(\frac{m2^k}{n}+o(2^k)\right)\log2}{2^{k+1}-1}\\
		&=\frac{m\log 2}{2n}.
	\end{align*}
	Thus, the closure of $\{\frac{m\log 2}{2n}: m,n\in\mathbb{N}\}$ is $[0,\infty)$. Since $h(X_1\times X_2)\geq 0$ for all SFTs $X_1$ and $X_2$, we have the closure of $\mathcal{H}^{2,0}$ is $[0,\infty)$. 

 Finally, we prove that the closure of $\mathcal{H}^{2,0}_i$ does not intersect $(0,\frac{\log 2}{2})$. We divide the proof into the following three cases: (1) If the transition matrix of $X$ contains no irreducible component, then the global pattern does not exist. Thus, $h(X\times X)=0$. (2) If the transition matrix of $X$ is irreducible and contains no row with a row sum greater than or equal to 2, then the number of global patterns is less than or equal to the number of symbols, that is 
	\[
	\left|\mathcal{P}(\Delta_n,X\times X)\right|\leq 2
	\]
	for all $n\geq 1$. Hence, we have $h(X\times X)=0$. (3) The transition matrix of $X$ is irreducible and contains a row $i$ with its row sum greater than or equal to 2. Then we put the symbol $i$ on $\mathcal{T}^{n-1}$, we have
	\begin{equation*}
		\left|\mathcal{P}(\Delta_n,X\times X)\right|\geq 2^{2^{n}}.
	\end{equation*}
	Thus, $h(X_1\times X_1)\geq \frac{\log 2}{2}$. Therefore, the proof of theorem is complete.
\end{proof}

\subsection{Transitivity and entropy}
Let $X\subseteq \mathcal{A}^\mathbb{N}$ be a subshift and $\sigma$ be the shift map. We say $X$ is \emph{transitive} if for any patterns $u,v$ of $X$, there exists an $n\in\mathbb{N}$ such that $\sigma^n([u])\cap [v]\neq \emptyset$, where $[w]$ is the cylinder set of the pattern $w$. A point $x\in X$ is called a \emph{periodic point} if there is a period $p\in\mathbb{N}$ such that $\sigma^p(x)=x$. 

Let $X\subseteq \mathcal{A}^{\mathbb{N}^2}$ be a subshift. The \emph{horizontal} shift map $\sigma_1$ and the \emph{vertical} shift map $\sigma_2$ are defined by $(\sigma_1(x))_{i,j}=(x)_{i+1,j}$ and $(\sigma_2(x))_{i,j}=(x)_{i,j+1}$ for all $i,j\geq 1,x\in X$. We say $X$ is \emph{transitive} if for any patterns $u,v$ of $X$, there exist $n,m\in\mathbb{N}$ such that $\sigma_1^n\sigma_2^m([u])\cap [v]\neq \emptyset$. A point $x\in X$ is called a \emph{periodic point} if there exist $p_1,p_2\in\mathbb{N}$ such that $\sigma_1^{p_1}\sigma_2^{p_2}(x)=x$. 

Let $T\subseteq \mathcal{A}^{\mathcal{T}_2}$ be a 2-tree-shift with the shift map $\sigma$. We say $T$ is \emph{transitive} if for any patterns $u,v$ of $T$, there exists a $w\in \mathcal{T}_2$ such that $\sigma^w([u])\cap [v]\neq \emptyset$. A point $t\in T$ is called a \emph{periodic point} if there is a $w\in \mathcal{T}_2$ such that $\sigma^{w}(t)=t$. 

The following theorem is directly obtained from Theorem \ref{thm 2.1}.
	
\begin{theorem}\label{thm inf}
	 We have $\inf\mathbf{H}^{2,0}$ (resp. $\inf\mathcal{H}^{2,0}$) equals 0, where the infimum is taken over all $X_1\times X_2$ (resp. $X_1\otimes X_2$) is transitive and has a periodic point.  
\end{theorem}

\begin{proof}
It is easy to check that the cases of $X_1\times X_2$ and $X_1 \otimes X_2$ in the proof of the Theorem \ref{thm 2.1} are transitive and have a periodic point. The proof is complete.
\end{proof}	

\section{The entropy structures of full axial extension shifts on $\mathbb{N}^{d}$ and $\mathcal{T}_{d}$}

In this section, we investigate the entropy structures of the full axial extension shifts on $\mathbb{N}^d$ and $\mathcal{T}_d$.

\subsection{Entropy structure}
In this subsection, we provide the rigorous entropy formulas for the full axial extension shifts on $\mathbb{N}^d$ and $\mathcal{T}_d$. In the following theroem, it is easy to see that the formulas are quite different.   

\begin{theorem}\label{thm 3.1}
	Let $d\geq 2$ and $d\geq r\geq 1$ be two positive integers. If $X_1,...,X_{d-r}$ are subshifts, then 
	\item[1.] $h(E^r\otimes_{i=1}^{d-r} X_i)=h(\otimes_{i=1}^{d-r} X_i)$, and
	\item[2.] $h(E^r\times_{i=1}^{d-r} X_i)=r(d-1)\sum_{j=1}^\infty \frac{\log\left|\mathcal{P}(\Delta_{j-1},\times_{i=1}^{d-r} X_i)\right|}{d^{j+1}}$.
\end{theorem}

\begin{proof}
	\item[\bf 1.] Due to the rule of $E^r\otimes_{i=1}^{d-r} X_i$, for any $n_1,...n_d\in\mathbb{N}$, the lattice $\mathbb{Z}_{n_1\times\cdots\times n_d}$ can be divided into $n_1\times\cdots\times n_r$ many $\mathbb{Z}_{n_{r+1}\times\cdots\times n_d}$ lattices in which \[\left|\mathcal{P}\left(\mathbb{Z}_{n_1\times\cdots\times n_d},E^r\otimes_{i=1}^{d-r} X_i\right)\right|=\left|\mathcal{P}\left(\mathbb{Z}_{n_{r+1}\times \cdots\times n_d},\otimes_{i=1}^{d-r} X_i\right)\right|^{n_1\cdots n_r}.\]
	 Then we have
		\begin{align*}
		h\left(E^r\otimes_{i=1}^{d-r} X_i\right)&=\lim_{n_1,...,n_d\to \infty}\frac{\log\left|\mathcal{P}\left(\mathbb{Z}_{n_1\times\cdots\times n_d},E^r\otimes_{i=1}^{d-r} X_i\right)\right|}{n_1\cdots n_d}\\
		&=\lim_{n_1,...,n_d\to \infty}\frac{\log\left|\mathcal{P}\left(\mathbb{Z}_{n_{r+1}\times \cdots\times n_d},\otimes_{i=1}^{d-r} X_i\right)\right|^{n_1\cdots n_r}}{n_1\cdots n_d}\\
		&=\lim_{n_{r+1},...,n_d\to\infty}\frac{\log\left|\mathcal{P}\left(\mathbb{Z}_{n_{r+1}\times \cdots\times n_d},\otimes_{i=1}^{d-r} X_i\right)\right|}{n_{r+1}\cdots n_d}\\
		&=h\left(\otimes_{i=1}^{d-r} X_i\right).
	\end{align*}
	\item[\bf 2.] Due to the rule of $E^r\times_{i=1}^{d-r} X_i$, for any $n\in\mathbb{N}$, the $\Delta_{n}$ of $d$-tree can be divided into a $\Delta_n$ of $(d-r)$-tree and $rd^{n-j}$ many $\Delta_{n-j}$ of $(d-r)$-tree for all $1\leq j \leq n$ in which
	\[
	\left|\mathcal{P}\left(\Delta_n,E^r\times_{i=1}^{d-r} X_i\right)\right|=\left|\mathcal{P}\left(\Delta_{n},\times_{i=1}^{d-r} X_i\right)\right|\prod_{j=1}^n\left|\mathcal{P}\left(\Delta_{j-1},\times_{i=1}^{d-r} X_i\right)\right|^{rd^{n-j}}.
	\]
	
	We prove the foregoing by induction. For $n=1$, the $\Delta_1$ of $d$-tree can be divided into $r$ many $\Delta_0$ and a $\Delta_1$ of $(d-r)$-tree in which
	\[
	\left|\mathcal{P}\left(\Delta_1,E^r\times_{i=1}^{d-r} X_i\right)\right|=\left|\mathcal{P}\left(\Delta_1,\times_{i=1}^{d-r} X_i\right)\right|\left|\mathcal{P}\left(\Delta_0,\times_{i=1}^{d-r} X_i\right)\right|^{r}.
	\]
	 Suppose it is true for $n=k$, then for $n=k+1$, the $\Delta_{k+1}$ of $d$-tree is the $\Delta_k$ of $d$-tree that add a $\mathcal{T}_{k+1}$ in the bottom. Then each partition of $\Delta_k$ of $d$-tree become to a partition of $\Delta_{k+1}$ of $d$-tree when it adds a level in the bottom. More precisely, the $\Delta_{j-1}$ of $(d-r)$-tree (in the partition of $\Delta_k$ of $d$-tree) becomes to a $\Delta_{j}$ of $(d-r)$-tree and $r(d-r)^{j-1}$ many $\Delta_0$ for all $1\leq j \leq k+1$ (in the partition of $\Delta_{k+1}$ of $d$-tree). Then, by the assumption that $\Delta_k$ of $d$-tree can be divided into a $\Delta_k$ of $(d-r)$-tree and $rd^{k-j}$ many $\Delta_{k-j}$ of $(d-r)$-tree for all $1\leq j \leq k$, we conclude that the $\Delta_{k+1}$ of $d$-tree can be divided into a $\Delta_{k+1}$ of $(d-r)$-tree, $rd^{k-j}$ many $\Delta_{k+1-j}$ of $(d-r)$-tree for all $1\leq j \leq k+1$ and $r(d-r)^k+\sum_{j=1}^k r(rd^{k-j})(d-r)^{j-1}=rd^k$ many $\Delta_0$ in which
	\[
	\left|\mathcal{P}\left(\Delta_{k+1},E^r\times_{i=1}^{d-r} X_i\right)\right|=\left|\mathcal{P}\left(\Delta_{k+1},\times_{i=1}^{d-r} X_i\right)\right|\prod_{j=1}^{k+1}\left|\mathcal{P}\left(\Delta_{j-1},\times_{i=1}^{d-r} X_i\right)\right|^{rd^{k+1-j}}.
	\]
	 The proof of induction is complete. See Figure \ref{inductionstep} for the induction step when $n=1$.  
	\begin{figure} 
		\centering 
		\includegraphics[width=1\textwidth]{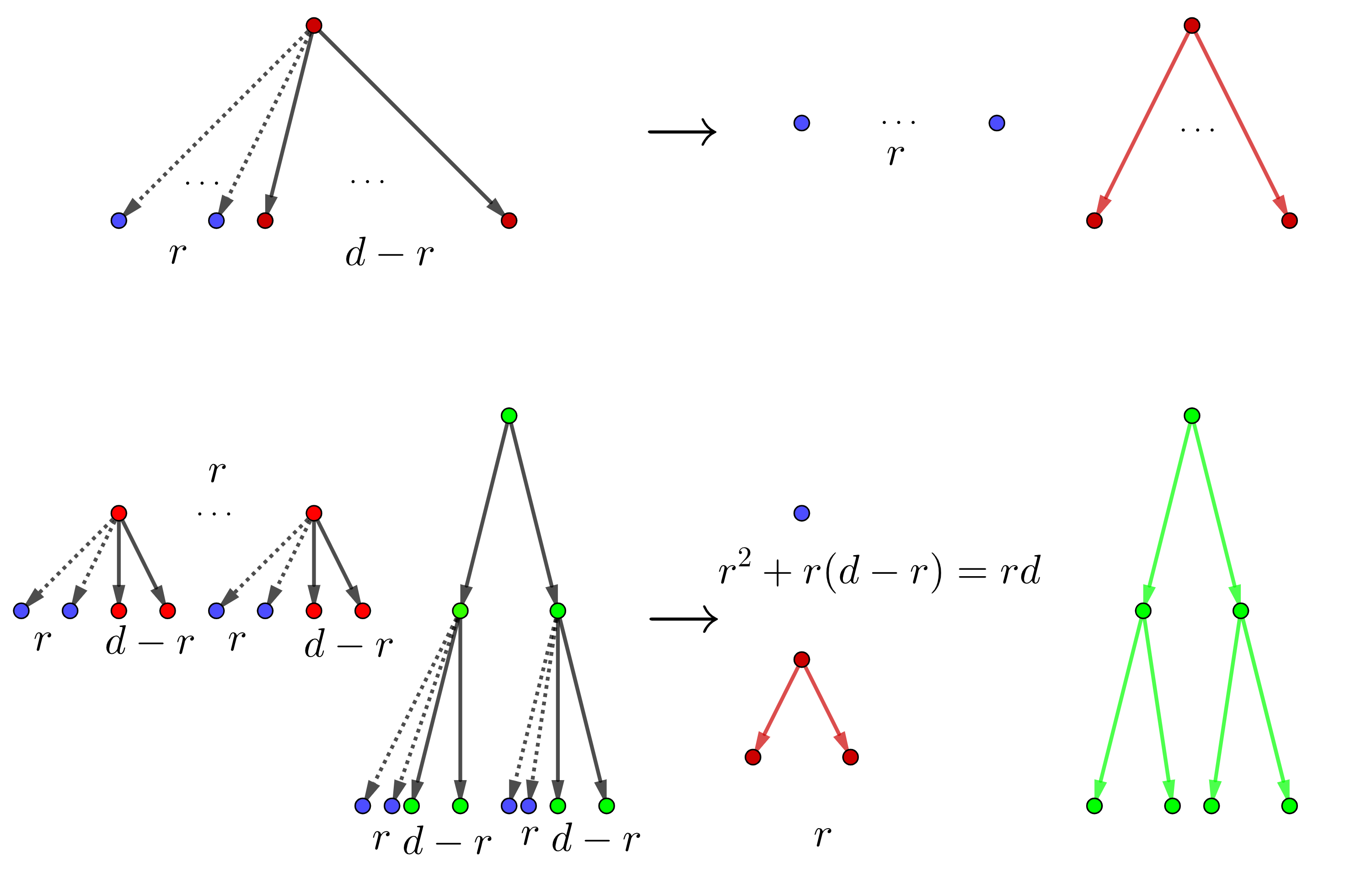} 
		\caption{The induction step for $d$-tree's $\Delta_1$ to $d$-tree's $\Delta_2$.} 
		\label{inductionstep} 
	\end{figure}

	Then we have the following two cases: (1) When $r=d$, we have
	\begin{align*}
		h\left(E^d \right)=\lim_{n\to \infty}\frac{\log\left|\mathcal{P}\left(\Delta_n,E^d\right)\right|}{\left|\Delta_n\right|}=\lim_{n\to \infty}\frac{\log 2^{\left|\Delta_{n}\right|}}{\left|\Delta_n\right|}=\log 2.
	\end{align*}
	(2) When $1\leq r<d$, since
	\begin{equation*}
	0\leq \lim_{n\to\infty}\frac{\log\left|\mathcal{P}\left(\Delta_n,\times_{i=1}^{d-r} X_i\right)\right|}{1+d+\cdots +d^n}\leq \lim_{n\to\infty}\frac{\log 2^{1+(d-r)+\cdots+(d-r)^n}}{1+d+\cdots +d^n}=0.
	\end{equation*}
	We have
	\begin{align*}
		h\left(E^r\times_{i=1}^{d-r} X_i\right)&=\lim_{n\to \infty}\frac{\log\left|\mathcal{P}\left(\Delta_n,E^r\times_{i=1}^{d-r} X_i\right)\right|}{\left|\Delta_n\right|}\\
		&=\lim_{n\to \infty}\frac{\log\left|\mathcal{P}\left(\Delta_{n},\times_{i=1}^{d-r} X_i\right)\right|\prod_{j=1}^n\left|\mathcal{P}\left(\Delta_{j-1},\times_{i=1}^{d-r} X_i\right)\right|^{rd^{n-j}}}{1+d+\cdots +d^n}\\
		&=\lim_{n\to \infty}\frac{\sum_{j=1}^n rd^{n-j}\log\left|\mathcal{P}\left(\Delta_{j-1},\times_{i=1}^{d-r} X_i\right)\right|}{1+d+\cdots +d^n}+\lim_{n\to\infty}\frac{\log\left|\mathcal{P}\left(\Delta_n,\times_{i=1}^{d-r} X_i\right)\right|}{1+d+\cdots +d^n}\\
		&=\lim_{n\to \infty}\frac{\sum_{j=1}^n rd^{-j}\log\left|\mathcal{P}\left(\Delta_{j-1},\times_{i=1}^{d-r} X_i\right)\right|}{1+d^{-1}+\cdots +d^{-n}}\\
		&=r(d-1)\sum_{j=1}^\infty \frac{\log\left|\mathcal{P}\left(\Delta_{j-1},\times_{i=1}^{d-r} X_i\right)\right|}{d^{j+1}}.
	\end{align*}
The proof is complete.
\end{proof}	
The set $\times_{i=1}^{d-r} X_i$ is called \emph{nonempty} if there exists a global pattern. For $k\in\mathbb{N}$, a \emph{$k$-permutation matrix} is a $k$ by $k$ binary matrix that has exactly one entry of 1 in each row and each column and 0 elsewhere. 

\begin{theorem}\label{thm 4.2} Let $\mathcal{A}=\{0,...,k-1\}$. For $d\geq 2$, $1\leq r \leq d$ and $\times_{i=1}^{d-r} X_i\subseteq \mathcal{A}^{\mathcal{T}_{d-r}}$, we have the following assertions.
	\item[1.] $\inf  h(E^r\times_{i=1}^{d-r} X_i)=\frac{r(d-1)\log k}{d^2}$, where the infimum is taken over all nonempty $\times_{i=1}^{d-r}X_i$.
	\item[2.] $\inf  h(E^r\times_{i=1}^{d-r} X_i)=\frac{r\log k}{d}$, where the infimum is taken over all subshift $\times_{i=1}^{d-r}X_i$.
	\item[3.] If $X_1,...,X_{d-r}$ are irreducible SFTs and $\times_{i=1}^{d-r} X_i$ is a subshift, then $h(E^r\times_{i=1}^{d-r} X_i)=\frac{r\log k}{d}$ if and only if $X_1,...,X_{d-r}$ are determined by $k$-permutation matrices.
\end{theorem}

\begin{proof}
	\item[\bf 1.] Since $|\mathcal{P}(\Delta_0,\times_{i=1}^{d-r} X_i)|=k$ and $|\mathcal{P}(\Delta_n,\times_{i=1}^{d-r} X_i)|\geq 1$ for all $n\geq 1$, then the proof is completed by the Theorem \ref{thm 3.1} (2).
	\item[\bf 2.] It is clear that if $\times_{i=1}^{d-r}X_i$ is a subshift, then $|\mathcal{P}(\Delta_n,\times_{i=1}^{d-r} X_i)|\geq k$ for all $n\geq 0$. Let $X_1=\cdots =X_{d-r}=\{0^{\infty},...,(k-1)^{\infty}\}$ be an subshift that contains only $k$ points, then $\times_{i=1}^{d-r} X_i$ is also an subshift that contains only $k$ points. Thus, $|\mathcal{P}(\Delta_n,\times_{i=1}^{d-r} X_i)|=k$ for all $n\geq 0$, then the proof is completed by the Theorem \ref{thm 3.1} (2). 
	\item[\bf 3.] For $d=2$ and $r=1$, if $A$ is a transition matrix of $X_1$, then $|A^n|=k$ for all $n\geq 1$ if and only if $A$ is a permutation matrix or union of permutation matrices with one sink. Then the irreducibility of $A$ implies $A$ is a $k$-permutation matrix. For $1<r<d$, it is easy to see that if $X_1,...,X_{d-r}$ are $k$-permutation matrices, then $|\mathcal{P}(\Delta_n,\times_{i=1}^{d-r} X_i)|=k$ for all $n\geq 0$. Conversely, if $|\mathcal{P}(\Delta_n,\times_{i=1}^{d-r} X_i)|=k$ for all $n\geq 0$, then for the root of $(d-r)$-tree labels $i\in\{0,...,k-1\}$, we have $\mathcal{P}(\Delta_0,\times_{i=1}^{d-r} X_i)=\{0,...,k-1\}$. This implies that if two global patterns $t_1$ and $t_2$ on $(d-r)$-tree with the same root symbol, then $t_1=t_2$. This implies that the symbol $i$ in each $X_1,...,X_{d-r}$ goes to unique $j\in \{0,1,...,k-1 \}$. Then the irreducibility implies that $X_{r+1},...,X_d$ are determined by $k$-permutation matrices. The proof is complete.
\end{proof}	

\subsection{Application: surface entropy}\label{sec surface}

Let $F_{n,m}=[1,n]\times \lbrack 1,m]\subseteq \mathbb{N}^{2}$ and $X$ be a
shift on $\mathbb{N}^{2}$, the \emph{surface entropy }of $X$ with
eccentricity $\alpha $ is 
\[
h_{s}(X,\alpha )=\sup_{\{(x_{n},y_{n})\}\in \Gamma _{\alpha
}}\limsup\limits_{n\rightarrow \infty }S_{X}(x_{n},y_{n})\text{,}
\]%
where 
\[
S_{X}(x_{n},y_{n})=\frac{\log \left\vert \mathcal{P}(F_{x_{n},y_{n}},X)%
	\right\vert -x_{n}y_{n}h(X)}{x_{n}+y_{n}}\text{,}
\]%
and $\Gamma _{\alpha }=\{\{(x_{n},y_{n})\}\in \left( \mathbb{N}^{2}\right) ^{%
	\mathbb{N}}:\frac{y_{n}}{x_{n}}\rightarrow \alpha $ and $x_{n}\rightarrow
\infty \}$. The main purpose of this study is to understand the `linear term'
of the complexity function $\log \left\vert \mathcal{P}(F_{x_{n},y_{n}},X)%
\right\vert $. This concept was first introduced by Pace \cite%
{pace2018surface}, wherein the author obtains the rigorous formula for $\mathbb{Z}$
SFTs and a lot of interesting properties for $\mathbb{Z}^{2}$ SFTs. Some
related results for the surface entropy may also be found in \cite%
{callard2021computational, meyerovitch2009growth}. Below we introduce the
concept of the $k$-multiplicative integer system ($k$-MIS) on $\mathbb{N}^{d}
$, and the main aim of this section is to combine the results of the surface
entropy for $k$-MIS and the full axial extension to derive the
surface entropy for the complexity function $\log\mathcal{P}(\Delta_n,E^{d-1}\times \Omega )$.

Let $\Omega \subseteq \mathcal{A}^{\mathbb{N}}$ be a $\mathbb{N}$ subshift
and $\mathbf{p}_{1},\ldots ,\mathbf{p}_{k-1}\in \mathbb{N}^{d}$, $d\geq 1$.
The $k$\emph{-multiplicative integer system} with respect to $\Omega $ is
defined as 
\begin{equation}
X_{\Omega }^{\mathbf{p}_{1},\ldots ,\mathbf{p}_{k-1}}=\{(x_{\mathbf{i}})_{%
	\mathbf{i}\in \mathbb{N}^{d}}\in \mathcal{A}^{\mathbb{N}^{d}}:x_{\mathbf{i}%
}x_{\mathbf{ip}_{1}}\cdots x_{\mathbf{ip}_{k-1}}\in \Omega _{k}\text{ }%
\forall \mathbf{i}\in \mathbb{N}^{d}\}\text{,}  \label{3}
\end{equation}%
where $\Omega _{k}$ denotes the set of admissible blocks of $\Omega$ with length $1<k\in 
\mathbb{N}$. The $X_{\Omega }^{\mathbf{p}_{1},\ldots ,\mathbf{p}_{k-1}}$ in (%
\ref{3}) is the multidimensional version of the multiplicative shift $%
X_{\Omega }^{p}$ defined by \cite{fan2012level, kenyon2012hausdorff}. 
\[
X_{\Omega }^{p}=\{(x_{k})_{k=1}^{\infty }\in \mathcal{A}^{\mathbb{N}%
}:(x_{ip^{l}})_{l=0}^{\infty }\in \Omega \text{ for all }i\text{, }p\nmid i\}%
\text{.}
\]%
The $k$-MIS or multiplicative shifts have been studied in depth because they
have a close relationship with multifractal analysis of the multiple ergodic
theory (cf. \cite{fan2014some}). We refer the reader to \cite{fan2014some,assani2014survey} for more details on the multiple ergodic theory and find
the complete bibliography. In \cite{ban2022boundary}, the authors consider the surface entropy
of $2$-MISs on $\mathbb{N}^{d},d\geq 1$.

\begin{theorem}[Theorem 3.3 with $d=1$ \cite{ban2022boundary}]\label{boundary}
	\item[1.] For a sequence $\{a_n\}_{n=1}^\infty$ with $\lim_{n\to\infty}a_n=\infty$, we have
	\begin{equation*}
	\log \left|\mathcal{P}\left(\mathbb{Z}_{x_n},X^{ p}_{\Omega}\right)\right|-x_nh=\left(1-\frac{1}{p}\right)^2\sum_{i=r_n+1}^{\infty}\frac{x_n}{p^{i-1}}\log |\Omega_i|+o\left(a_nx_n\right),
	\end{equation*}	
	where $r_n=\left\lfloor \frac{\log x_n}{\log p} \right\rfloor$.
	\item[2.] If $x_n=p^n\pm k$ and $1\leq k \leq p$, then we have 
	\begin{equation*}
	\log \left|\mathcal{P}\left(\mathbb{Z}_{x_n},X^{ p}_{\Omega}\right)\right|-x_nh=O\left(n\right).	
	\end{equation*}
\end{theorem}

\begin{theorem}
	\item[1.] We have
	\begin{equation*}
	h\left(E^{d-1}\times X\right)=(d-1)^2\sum_{i=1}^\infty \frac{\log\left|\mathcal{P}\left(\mathbb{Z}_i,X\right)\right|}{d^{i+1}}.
	\end{equation*}
	\item[2.] For $d\geq 2$, let $\Omega$ be a subshift, we have 
	\begin{equation*}
	\log\left|\mathcal{P}\left(\Delta_n,E^{d-1}\times \Omega\right) \right|=\left|\Delta_n \right|h(E^{d-1}\times \Omega)+(d-1)^2\sum_{i=r_n+1}^{\infty}\frac{|\Delta_n|}{d^{i-1}}\log |\mathcal{P}(\mathbb{Z}_i,\Omega)|+o\left(a_n|\Delta_n|\right),
	\end{equation*}
	where $r_n=\left\lfloor \frac{\log |\Delta_n|}{\log d} \right\rfloor$ and $\{a_n\}_{n=1}^\infty$ is a sequence with $\lim_{n\to\infty}a_n=\infty$.
	
	 In particular, when $d=2$, we have
	 \begin{equation*}
	 \log\left|\mathcal{P}\left(\Delta_n,E\times \Omega\right) \right|=\left|\Delta_n \right|h(E\times \Omega)+O(n).
	 \end{equation*}
\end{theorem}

\begin{proof}
	\item[\bf 1.] By Theorem \ref{thm 3.1} with $r=d-1$, we have
	\[
	h\left(E^{d-1}\times X\right)=(d-1)^2\sum_{i=1}^\infty \frac{\log\left|\mathcal{P}(\Delta_{i-1}, X)\right|}{d^{i+1}}.
	\]
	Since $\mathcal{P}(\Delta_{i-1},X)=\mathcal{P}(\mathbb{Z}_i,X)$ for all $i\geq 1$, the proof is thus complete.
	\item[\bf 2.] When $d\geq 2$, the proof is complete by taking $x_n=|\Delta_n|$ and $p=d$ in Theorem \ref{boundary} (1). When $d=2$, since $|\Delta_n|=2^{n+1}-1$, the proof is complete by applying Theorem \ref{boundary} (2).
\end{proof}

\section{General trees}\label{sec 5}

In this section, we study the entropy of the full axial extension on general trees. Let $d\geq 2$, the $d$-tree $\mathcal{T}_d=\langle f_1,...,f_d\rangle$ is a semigroup which is generated by $d$ generators $f_1,...,f_d$. For any $g\in\mathcal{T}_d$ and a binary matrix $M=[M_{i,j}]_{d\times d}$, we say $g$ is \emph{admissible} with respect to $M$ if $g=f_{a_1}\cdots f_{a_k},k\geq 1,a_1,...,a_k\in\{1,...,d\}$ satisfies $M_{a_j,a_{j+1}}=1$ for all $1\leq j \leq k-1$. Define the \emph{Markov-Cayley tree} $\mathcal{T}$ with the adjacency matrix $M=[M_{i,j}]_{d\times d}$ is the set $\left\{ g\in\mathcal{T}_d: g\mbox{ is admissible w.r.t. }M \right\}$. The \emph{golden-mean tree} $\mathcal{G}$ is defined by the adjacency matrix $M=\left[\begin{array}{cc}
	1 & 1 \\ 
	1 & 0%
\end{array}\right]$. That is, $\mathcal{G}=\left\{g\in \mathcal{T}_2: g\mbox{ is admissible w.r.t. }\left[\begin{array}{cc}
1 & 1 \\ 
1 & 0%
\end{array}\right] \right\}$.

The following result demonstrates that the values $h(E\times X)$ and $h(X\times E)$ are not general coincident.
\begin{theorem}[Golden-mean tree]\label{thm 5.1}
	If $X$ is a subshift, then the entropy of $E\times X$ and $X\times E$ on $\mathcal{G}$ are equal to $\frac{\log\left|\mathcal{P}(\mathbb{Z}_1,X)\right|}{\rho^3} +\frac{\log\left|\mathcal{P}(\mathbb{Z}_2,X)\right|}{\rho^2}$ and $\sum_{i=1}^\infty \frac{\log \left|\mathcal{P}(\mathbb{Z}_i,X)\right|}{\rho^{i+3}}$ respectively, where $\rho=\frac{1+\sqrt{5}}{2}$.
\end{theorem}	

\begin{proof}
	For the entropy of $E\times X$, we claim that for $n\geq 2$, $\Delta_n^{\mathcal{G}}$ can be divided into $a_n$ many $\mathbb{Z}_1$ lattices and $\frac{\sum_{i=1}^{n+1} a_i-a_n}{2}$ many $\mathbb{Z}_2$ lattices in which
	\[
	\left|\mathcal{P}\left(\Delta^{\mathcal{G}}_{n},E\times X\right)\right|=\left|\mathcal{P}(\mathbb{Z}_1,X)\right|^{a_n}\left|\mathcal{P}\left(\mathbb{Z}_2,X\right)\right|^{\frac{\sum_{i=1}^{n+1} a_i-a_n}{2}},
	\]
	 where $\Delta^{\mathcal{G}}_n=\{g\in\mathcal{G}: |g|\leq n\}$ is the $n$-block of the golden-mean tree and $a_{n+2}=a_{n+1}+a_{n}$ with $a_1=1$ and $a_2=2$. Indeed, due to the rule of $E\times X$, we divide the $\Delta^{\mathcal{G}}_n$ into $|\mathcal{T}_{n-1}^\mathcal{G}|$ many $\mathbb{Z}_1$ lattices and all the other lattices of the partition of $\Delta^{\mathcal{G}}_n$ are $\mathbb{Z}_2$ in which
	 \[
	 \left|\mathcal{P}\left(\Delta^{\mathcal{G}}_{n},E\times X\right)\right|=\left|\mathcal{P}(\mathbb{Z}_1,X)\right|^{|\mathcal{T}_{n-1}^\mathcal{G}|}\left|\mathcal{P}\left(\mathbb{Z}_2,X\right)\right|^{\frac{\left|\Delta_n^{\mathcal{G}}\right|-\left|\mathcal{T}_{n-1}^\mathcal{G}\right|}{2}},
	 \]
	 where $\mathcal{T}_n^\mathcal{G}=\{g\in\mathcal{G}: |g|= n\}$. The proof of claim is completed by $\left|\Delta_n^{\mathcal{G}}\right|=\sum_{i=1}^{n+1}a_i$ and $\left|\mathcal{T}_{n-1}^\mathcal{G}\right|=a_n$. By the construction of the partitions based on $X\times E$ (See the left hand side of Figure \ref{partitiongm} for the partitions of $\Delta_1^\mathcal{G}$, $\Delta_2^\mathcal{G}$ and $\Delta_3^\mathcal{G}$ based on $E\times X$.), thus we have 
	\begin{figure} 
		\centering 
		\includegraphics[width=0.8\textwidth]{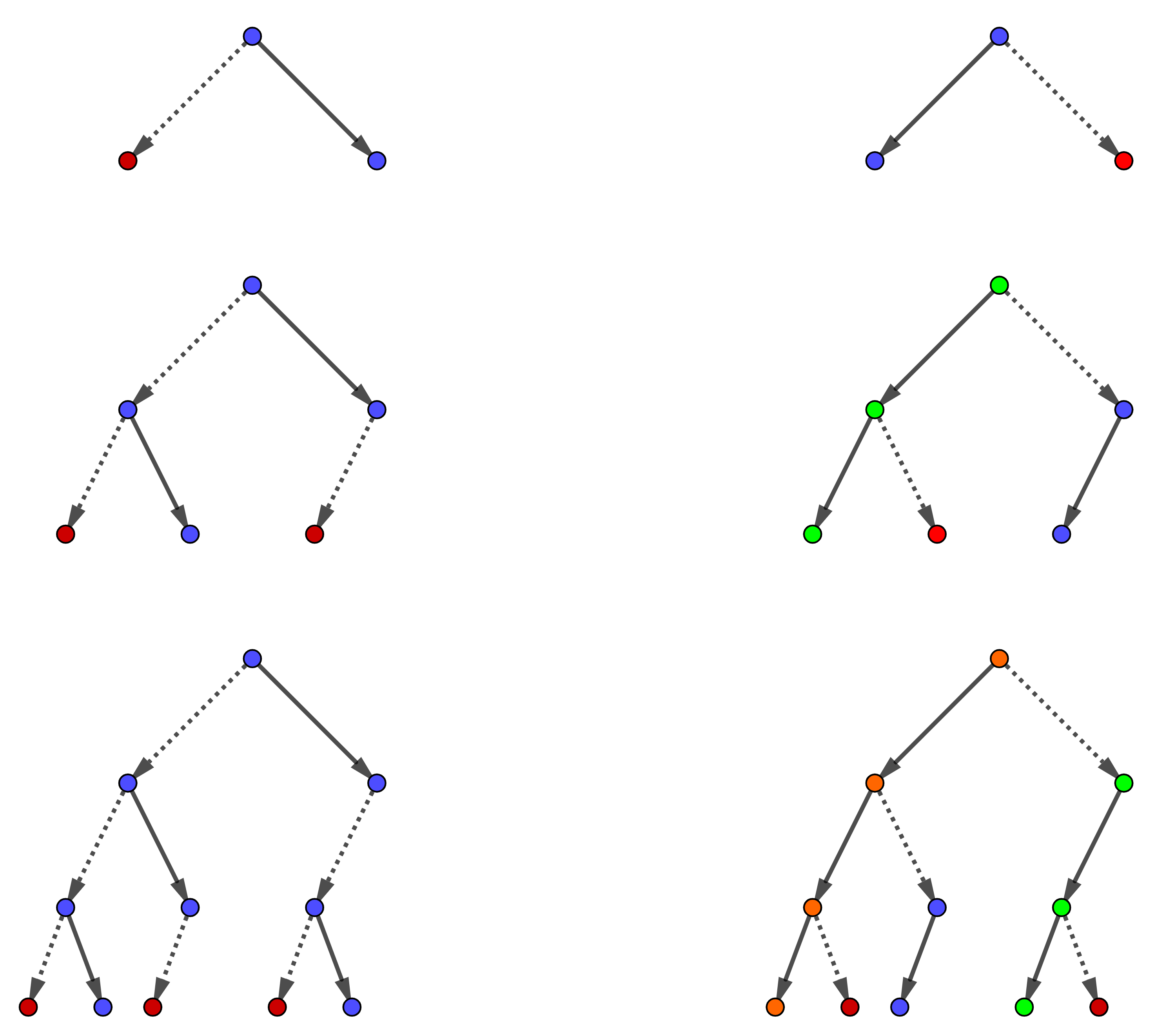} 
		\caption{The left (resp. right) hand side illustrate the partition of $\Delta_1^\mathcal{G}$, $\Delta_2^\mathcal{G}$ and $\Delta_3^\mathcal{G}$ based on $E\times X$ (resp. $X\times E$).} 
		\label{partitiongm} 
	\end{figure}
	 \begin{align*}
	h\left(E\times X\right)&=\lim_{n\to \infty}\frac{\log\left|\mathcal{P}\left(\Delta^{\mathcal{G}}_{n},E\times X\right)\right|}{\left|\Delta^{\mathcal{G}}_n\right|}\\
	&=\lim_{n\to \infty}\frac{\log\left|\mathcal{P}(\mathbb{Z}_1,X)\right|^{a_n}\left|\mathcal{P}\left(\mathbb{Z}_2,X\right)\right|^{\frac{\sum_{i=1}^{n+1} a_i-a_n}{2}}}{a_1+a_2+\cdots +a_{n+1}}\\
	&=\lim_{n\to \infty}\frac{a_n\log\left|\mathcal{P}(\mathbb{Z}_1,X)\right|+\frac{1}{2}\left(\sum_{i=1}^{n+1} a_i-a_n\right)\log\left|\mathcal{P}\left(\mathbb{Z}_2,X\right)\right|}{a_1+a_2+\cdots +a_{n+1}}\\
	&=\lim_{n\to \infty}\frac{a_n\log\left|\mathcal{P}(\mathbb{Z}_1,X)\right|+\frac{1}{2}\left(a_{n+3}-2-a_n\right)\log\left|\mathcal{P}\left(\mathbb{Z}_2,X\right)\right|}{a_{n+3}-2}\\
	&=\lim_{n\to \infty}\frac{a_n\log\left|\mathcal{P}(\mathbb{Z}_1,X)\right|+\left(a_{n+1}-1\right)\log\left|\mathcal{P}\left(\mathbb{Z}_2,X\right)\right|}{a_{n+3}-2}\\
	&=\frac{1}{\rho^3}\log \left|\mathcal{P}(\mathbb{Z}_1,X)\right|+\frac{1}{\rho^2}\log\left|\mathcal{P}\left(\mathbb{Z}_2,X\right)\right|,
\end{align*}
where $\rho=\frac{1+\sqrt{5}}{2}$.

For the entropy of $X\times E$, we claim that for $n\geq 2$, $\Delta_n^{\mathcal{G}}$ can be divided into a $\mathbb{Z}_{n+1}$ lattice, a $\mathbb{Z}_n$ lattice and $a_{n-i}$ many $\mathbb{Z}_i$ lattices for all $1\leq i \leq n-1$ in which
\[
\left|\mathcal{P}\left(\Delta^{\mathcal{G}}_{n},X\times E\right)\right|=\left|\mathcal{P}\left(\mathbb{Z}_{n+1},X\right)\right|\left|\mathcal{P}\left(\mathbb{Z}_n,X\right)\right|\prod_{i=1}^{n-1}\left|\mathcal{P}\left(\mathbb{Z}_i,X\right)\right|^{a_{n-i}}.
\]
 Indeed, due to the rule of $E\times X$, for $n=2$, we divide the $\Delta_2^{\mathcal{G}}$ into a $\mathbb{Z}_3$ lattice, a $\mathbb{Z}_2$ lattice and $|\mathcal{T}_{0}^{\mathcal{G}}|=a_1=1$ many $\mathbb{Z}_1$ lattices in which
 \[
 \left|\mathcal{P}\left(\Delta^{\mathcal{G}}_{2},X\times E\right)\right|=\left|\mathcal{P}\left(\mathbb{Z}_{3},X\right)\right|\left|\mathcal{P}\left(\mathbb{Z}_2,X\right)\right|\left|\mathcal{P}\left(\mathbb{Z}_1,X\right)\right|^{a_1}.
 \]
 Assume the claim holds for $n=k$, that is, we divide $\Delta_k^{\mathcal{G}}$ into a $\mathbb{Z}_{k+1}$ lattice, a $\mathbb{Z}_k$ lattice and $a_{k-i}$ many $\mathbb{Z}_i$ lattices for all $1\leq i \leq k-1$ in which
 \[
 \left|\mathcal{P}\left(\Delta^{\mathcal{G}}_{k},X\times E\right)\right|=\left|\mathcal{P}\left(\mathbb{Z}_{k+1},X\right)\right|\left|\mathcal{P}\left(\mathbb{Z}_k,X\right)\right|\prod_{i=1}^{k-1}\left|\mathcal{P}\left(\mathbb{Z}_i,X\right)\right|^{a_{k-i}}.
 \]
 Since the $\mathbb{Z}_i$ of the partition of $\Delta_k^{\mathcal{G}}$ can be extended to the $\mathbb{Z}_{i+1}$ of the partition of $\Delta_{k+1}^{\mathcal{G}}$ for all $1\leq i\leq k+1$, and the number of $\mathbb{Z}_1$ of the partition of $\Delta_j^{\mathcal{G}}$ is $a_{j-2}$ for all $j\geq 2$. Then by assumption of claim, the $\Delta_{k+1}^{\mathcal{G}}$ can be divided into a $\mathbb{Z}_{k+2}$ lattice, a $\mathbb{Z}_{k+1}$ lattice and $a_{k+1-i}$ many $\mathbb{Z}_i$ lattices for all $1\leq i \leq k$ in which
 \[
 \left|\mathcal{P}\left(\Delta^{\mathcal{G}}_{k},X\times E\right)\right|=\left|\mathcal{P}\left(\mathbb{Z}_{k+2},X\right)\right|\left|\mathcal{P}\left(\mathbb{Z}_{k+1},X\right)\right|\prod_{i=1}^{k}\left|\mathcal{P}\left(\mathbb{Z}_i,X\right)\right|^{a_{k+1-i}}.
 \] 
  The proof of claim is complete. By the construction of the partitions based on $X\times E$ (See the right hand side of Figure \ref{partitiongm} for the partitions of $\Delta_1^\mathcal{G}$, $\Delta_2^\mathcal{G}$ and $\Delta_3^\mathcal{G}$ based on $X\times E$.), thus we have 
\begin{align*}
	h\left(X\times E\right)&=\lim_{n\to \infty}\frac{\log\left|\mathcal{P}\left(\Delta^{\mathcal{G}}_{n},X\times E\right)\right|}{\left|\Delta^{\mathcal{G}}_n\right|}\\
	&=\lim_{n\to \infty}\frac{\log\left|\mathcal{P}\left(\mathbb{Z}_{n+1},X\right)\right|\left|\mathcal{P}\left(\mathbb{Z}_n,X\right)\right|\prod_{i=1}^{n-1}\left|\mathcal{P}\left(\mathbb{Z}_i,X\right)\right|^{a_{n-i}}}{a_1+a_2+\cdots +a_{n+1}}\\
	&=\lim_{n\to \infty}\frac{\sum_{i=1}^{n-1}a_{n-i}\log\left|\mathcal{P}\left(\mathbb{Z}_i,X\right)\right|}{a_{n+3}-2}\\
	&=\sum_{i=1}^\infty \frac{\log \left|\mathcal{P}\left(\mathbb{Z}_i,X\right)\right|}{\rho^{i+3}}.
\end{align*}	
The proof is complete.	
\end{proof}	

\begin{remark}
We remark that if $X_A$ is a full shift with $A=\left[\begin{matrix}
1&1\\
1&1
\end{matrix} \right]$, then $\left|\mathcal{P}\left(\mathbb{Z}_i,X\right)\right|=2^i$ for all $i\geq 1$. This implies $h(E\times X)=h(X\times E)=h(X)=\log 2$. However, the equalities do not hold in general. If we let $X=X_A$ be an SFT with $A=\left[\begin{matrix}
1&0\\
0&1
\end{matrix} \right]$, then $\left|\mathcal{P}\left(\mathbb{Z}_i,X\right)\right|=2$ for all $i\geq 1$. This implies $h(E\times X)=(\frac{1}{\rho^3}+\frac{1}{\rho^2})\log 2$, $h(X\times E)=\frac{1}{\rho^2}\log 2$ and $h(X)=0$.
\end{remark}

In the following , we consider $h(E\times X)$ and $h(X\times E)$ on the Markov-Cayley trees.
	
\begin{theorem}[General trees]\label{thm general}	
	\item[1.] If the Markov-Cayley tree $\mathcal{T}$ is a subtree of $\mathcal{T}_2$ satisfies $\gamma >1$ where $\gamma=\lim_{n\to\infty}\frac{|\mathcal{T}_{n+1}|}{|\mathcal{T}_n|}$ and $\mathcal{T}_n=\{g\in \mathcal{T}: |g|=n\}$, then $h(E\times X)>h(X)$ and $h(X\times E)>h(X)$ when $X$ is not a full shift.
	\item[2.] There is a Markov-Cayley tree $\mathcal{T}$ that satisfies $\gamma=1$ such that $h(E\times X)=h(X)$ for all $X=X_A$ that is an SFT with primitive transition matrix $A$.
\end{theorem}	
\begin{proof}
\item[\bf 1.] Since $\gamma >1$ and $\mathcal{T}$ is a subtree of $\mathcal{T}_2$, the ratio of 2 branching $p_n=\frac{|\{g\in \mathcal{T}_n: gf_1,gf_2\in \mathcal{T}_{n+1} \}|}{|\mathcal{T}_n|}$ has a limit $p=\lim_{n\to\infty}p_n$ that is bounded away from zero. Due to the rule of $E\times X$, the $\Delta_n$ of $\mathcal{T}$ can be divided into $a_{n;i}$ many $\mathbb{Z}_i$, $1\leq i \leq n$ in which
\[
| \mathcal{P}(\Delta_n^\mathcal{T},E\times X)|=\prod_{i=1}^n |\mathcal{P}(\mathbb{Z}_i,X)|^{a_{n;i}},
\]
 where $a_{n;1}\geq |\{g\in\mathcal{T}_{n-1}: gf_1,gf_2\in\mathcal{T}_n\}|$. 

Thus, 
\begin{align*}
h(E\times X)&=\limsup_{n\to\infty} \frac{\log| \mathcal{P}(\Delta_n^\mathcal{T},E\times X)|}{|\Delta_n^\mathcal{T}|}\\
&=\limsup_{n\to\infty} \frac{\log\prod_{i=1}^n |\mathcal{P}(\mathbb{Z}_i,X)|^{a_{n;i}}}{|\Delta_n^\mathcal{T}|}\\
&=\limsup_{n\to\infty} \sum_{i=1}^n\frac{ \log |\mathcal{P}(\mathbb{Z}_i,X)|^{a_{n;i}}}{|\Delta_n^\mathcal{T}|}.
\end{align*} 
Since $\log |\mathcal{P}(\mathbb{Z}_n,X)|\geq nh(X)$ for all $n\geq 1$, we have
\begin{align*}
h(E\times X)\geq\limsup_{n\to\infty} \sum_{i=1}^n\frac{ a_{n;i} nh(X)}{|\Delta_n^\mathcal{T}|}=h(X).
\end{align*} 

Since $X$ is not a full shift, it is clear that $\log|\mathcal{P}(\mathbb{Z}_1,X)|>h(X)$. Then, it is enought to show that $\lim_{n\to\infty} \frac{a_{n;1}}{|\Delta_n^\mathcal{T}|}>0$ for proving $h(E\times X)>h(X)$. Indeed, it is equivalent to show the following limit exists and positive. We now claim
 \[
\lim_{n\to\infty}\frac{|\mathcal{T}_n|}{\sum_{i=1}^n |\mathcal{T}_i|}p_n=\frac{\gamma-1}{\gamma}p.
\]
For any $0<\epsilon<\frac{1}{2\gamma}$, since $\frac{1}{\gamma}=\lim_{n\to\infty}\frac{|\mathcal{T}_{n}|}{|\mathcal{T}_{n+1}|}$, there is an $m\in\mathbb{N}$ such that $\frac{1}{\gamma}-\epsilon<\frac{|\mathcal{T}_n|}{|\mathcal{T}_{n+1}|}<\frac{1}{\gamma}+\epsilon$ for all $n\geq m$. Then for $n\geq m$, we have    
\begin{align*}
&\frac{1}{m(\frac{1}{\gamma}+\epsilon)^{n-m}+\frac{\gamma}{\gamma-1-\gamma\epsilon}}\leq \frac{1}{m\frac{|\mathcal{T}_m|}{|\mathcal{T}_n|}+\sum_{i=0}^{\infty}(\frac{1}{\gamma}+\epsilon)^i}\\
&\leq
\frac{|\mathcal{T}_n|}{\sum_{i=1}^n |\mathcal{T}_i|}=\frac{1}{\sum_{i=1}^{m-1}\frac{|\mathcal{T}_i|}{|\mathcal{T}_n|}+\sum_{i=m}^n\frac{|\mathcal{T}_i|}{|\mathcal{T}_n|}}\\
&\leq \frac{1}{\sum_{i=1}^{m-1}\frac{|\mathcal{T}_i|}{|\mathcal{T}_n|}+\sum_{i=0}^{n-m}(\frac{1}{\gamma}-\epsilon)^i}\leq \frac{1}{\sum_{i=0}^{n-m}(\frac{1}{\gamma}-\epsilon)^i}.
\end{align*}
Taking $n\to\infty$, we have
\[
\frac{\gamma-1-\gamma\epsilon}{\gamma} \leq\lim_{n\to\infty}\frac{|\mathcal{T}_n|}{\sum_{i=1}^n |\mathcal{T}_i|}\leq \frac{\gamma-1+\gamma\epsilon}{\gamma}.
\]
Since $\epsilon$ is arbitrary, we obtain $\lim_{n\to\infty}\frac{|\mathcal{T}_n|}{\sum_{i=1}^n |\mathcal{T}_i|}= \frac{\gamma-1}{\gamma}$. The proof of claim is complete.

Therefore, we have $h(E\times X)> h(X)$. The proof of $h(X\times E)> h(X)$ is similar and thus we omit it.
\item[\bf 2.] Let $\mathcal{G}_1$ be a Markov-Cayley tree with the adjacency matrix $M=\left[\begin{matrix} 1&1\\0&1\end{matrix}\right]$. Due to the rule of $E\times X$, for $n\geq 2$, the $\Delta_n^{\mathcal{G}_1}$ can be divided into a $\mathbb{Z}_1$, a $\mathbb{Z}_2$, ... and a $\mathbb{Z}_{n+1}$ in which
\[
\left|\mathcal{P}\left(\Delta^{\mathcal{G}_1}_{n},E\times X\right)\right|=\prod_{i=1}^{n+1}\left|\mathcal{P}\left(\mathbb{Z}_i,X\right)\right|.
\]
 See Figure \ref{partitionM} for the partitions of $\Delta_1^{\mathcal{G}_1}$, $\Delta_2^{\mathcal{G}_1}$ and $\Delta_3^{\mathcal{G}_1}$ based on $E\times X$. 
\begin{figure} 
	\centering 
	\includegraphics[width=1\textwidth]{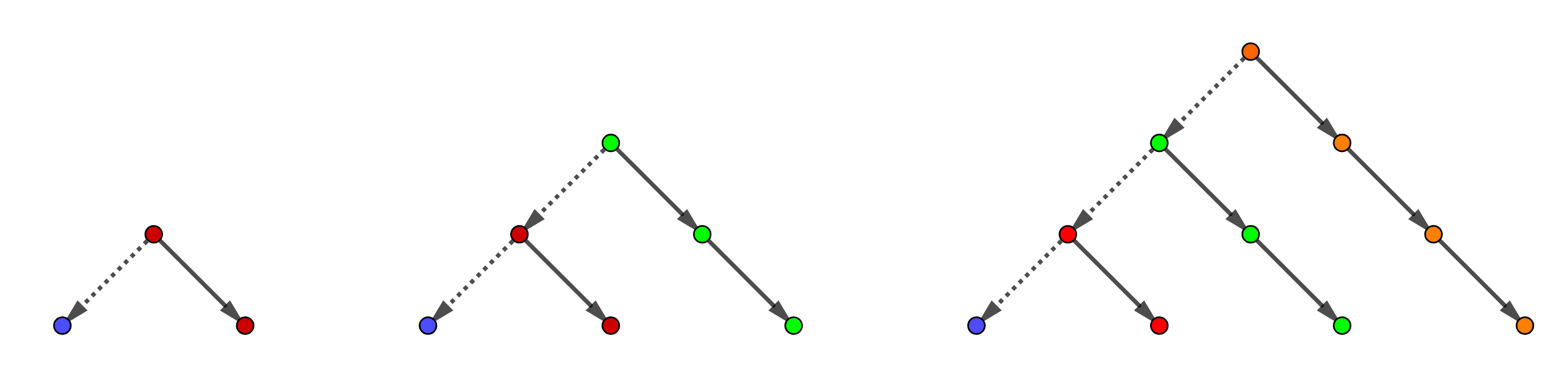} 
	\caption{The partitions of $\Delta_1^{\mathcal{G}_1}$, $\Delta_2^{\mathcal{G}_1}$ and $\Delta_3^{\mathcal{G}_1}$ based on $E\times X$.} 
	\label{partitionM} 
\end{figure}

Thus, we have the entropy of $E\times X$ on $\mathcal{G}_1$ is equal to 
\begin{align*}
	h(E\times X)&=\lim_{n\to \infty}\frac{\log\left|\mathcal{P}\left(\Delta^{\mathcal{G}_1}_{n},E\times X\right)\right|}{|\Delta^{\mathcal{G}_1}_n|}=\lim_{n\to \infty}\frac{\log \prod_{i=1}^{n+1}\left|\mathcal{P}\left(\mathbb{Z}_i,X\right)\right|}{1+2+\cdots +(n+1)}=\lim_{n\to \infty}\frac{ \sum_{i=1}^{n+1}\log\left|A^i\right|}{\frac{(n+1)(n+2)}{2}}.	
\end{align*}	
Assume $A$ is a primitive matrix and let $X=X_A$, there are constants $c_1$ and $c_2$ such that $c_1\lambda_A^i \leq |A^i| \leq c_2 \lambda_A^i$ for all $i\geq 1$. Then we have
\begin{align*}
	\log\lambda_A=\lim_{n\to \infty}\frac{ \sum_{i=1}^{n+1}\log c_1\lambda_A^i}{\frac{(n+1)(n+2)}{2}}\leq h(E\times X)&\leq \lim_{n\to \infty}\frac{ \sum_{i=1}^{n+1}\log c_2\lambda_A^i}{\frac{(n+1)(n+2)}{2}}=\log\lambda_A.	
\end{align*}
The proof is complete by $h(X)=\log \lambda_A$.
\end{proof}

\bibliographystyle{amsplain}
\bibliography{ban}

\end{document}